\documentclass{amsart}
\usepackage{latexsym, amssymb, amsmath, amsthm}
\usepackage[T1]{fontenc}
\usepackage{lmodern}
\usepackage{enumerate}
\usepackage{mathabx}
\usepackage{csquotes}
\usepackage[mathscr]{euscript}

\usepackage{tikz}
\usetikzlibrary{trees}

\usepackage{pifont}

\newtheorem{theorem}{Theorem}[section]
\newtheorem{lemma}[theorem]{Lemma}
\newtheorem{corollary}[theorem]{Corollary}

\newtheorem{question}[theorem]{Question}
\newtheorem{conjecture}[theorem]{Conjecture}
\newtheorem{claim}[theorem]{Claim}

\theoremstyle{definition}
\newtheorem{definition}[theorem]{Definition}
\theoremstyle{remark}
\newtheorem{remark}[theorem]{Remark}

\newbox\gnBoxA
\newdimen\gnCornerHgt
\setbox\gnBoxA=\hbox{$\ulcorner$}
\global\gnCornerHgt=\ht\gnBoxA
\newdimen\gnArgHgt
\def\gnmb #1{%
\setbox\gnBoxA=\hbox{$#1$}%
\gnArgHgt=\ht\gnBoxA%
\ifnum     \gnArgHgt<\gnCornerHgt \gnArgHgt=0pt%
\else \advance \gnArgHgt by -\gnCornerHgt%
\fi \raise\gnArgHgt\hbox{$\ulcorner$} \box\gnBoxA %
\raise\gnArgHgt\hbox{$\urcorner$}}

\title{A note on the consistency operator}

\author{James Walsh}
\address{Group in Logic and the Methodology of Science, University of California, Berkeley}
\email{walsh@math.berkeley.edu}

\thanks{2010 \emph{Mathematics Subject Classification}. Primary 03F40.}

\thanks{Thanks both to Antonio Montalb{\'a}n and to an anonymous referee for their comments and suggestions.}

\begin{document}

\maketitle

\begin{abstract}
It is a well known empirical observation that natural axiomatic theories are pre-well-ordered by proof-theoretic strength. For any natural theory $T$, the next strongest natural theory is $T+\mathsf{Con}_T$. We formulate and prove a statement to the effect that the consistency operator is the weakest natural way to uniformly extend axiomatic theories.
\end{abstract}

\section{Introduction}

G{\"o}del's second incompleteness theorem states that no consistent sufficiently strong effectively axiomatized theory $T$ proves its own consistency statement $\mathsf{Con}_T$. Using ad hoc proof-theoretic techniques (namely, Rosser-style self-reference) one can construct $\Pi_1$ sentences $\varphi$ that are not provable in $T$ such that $T+\varphi$ is a strictly weaker theory than $T+\mathsf{Con}_T$. Nevertheless, $\mathsf{Con}_T$ seems to be the weakest \emph{natural} $\Pi_1$ sentence that is not provable in $T$. Without a mathematical definition of ``natural,'' however, it is difficult to formulate a precise conjecture that would explain this phenomenon. This is a special case of the well known empirical observation that natural axiomatic theories are pre-well-ordered by consistency strength, which S. Friedman, Rathjen, and Weiermann \cite{friedman2013slow} call one of the ``great mysteries in the foundations of mathematics.''


Recursion theorists have observed a similar phenomenon in Turing degree theory. One can use ad hoc recursion-theoretic methods like the priority method to construct non-recursive $\Sigma_1$ definable sets whose Turing degree is strictly below that of $0'$. Nevertheless, $0'$ seems to be the weakest \emph{natural} non-recursive r.e. degree. Once again, without a mathematical definition of ``natural,'' however, it is difficult to formulate a precise conjecture that would explain this phenomenon. 

A popular approach to studying natural Turing degrees is to focus on degree-invariant functions; a function $f$ on the reals is \emph{degree-invariant} if, for all reals $A$ and $B$, $A\equiv_T B$ implies $f(A) \equiv_T f(B)$. The definitions of natural Turing degrees tend to relativize to arbitrary degrees, yielding degree invariant functions on the reals; for instance, the construction of $0'$ relativizes to yield the Turing Jump. Sacks \cite{sacks1963degrees} asked whether there is a degree invariant solution to Post's Problem. Recall that a function $W:2^\omega\rightarrow 2^\omega$ is a \emph{recursively enumerable operator} if there is an $e\in\omega$ such that, for each $A$, $W(A)=W_e^A$, the $e^{th}$ set recursively enumerable in $A$.

\begin{question}[Sacks]
Is there a degree-invariant recursively enumerable operator $W$ such that for every real $A$, $A <_T W_e^A <_T A'$?
\end{question}
Though the question remains open, Slaman and Steel \cite{slaman1988definable} proved that there is no order-preserving solution to Post's Problem. Recall that a function $f$ on the reals is \emph{order-preserving} if, for all reals $A$ and $B$, $A\leq_T B$ implies $f(A) \leq_T f(B)$. 





In \cite{montalban2019inevitability}, Montalb{\'a}n and the author proved a proof-theoretic analogue of a negative answer to Sacks' question for order-preserving functions. Let $T$ be a sound, sufficiently strong effectively axiomatized theory in the language of arithmetic, e.g., $\mathsf{EA}$.\footnote{$\mathsf{EA}$ is a theory in the language of arithmetic (with exponentiation) axiomatized by the axioms of Robinson's $Q$, recursive axioms for exponentiation, and induction for bounded formulas.} A function $\mathfrak{g}$ is \emph{monotone} if, for all sentences $\varphi$ and $\psi$, $T\vdash\varphi \rightarrow \psi$ implies $T\vdash \mathfrak{g}(\varphi)\rightarrow \mathfrak{g}(\psi)$ (this is just to say that $\mathfrak{g}$ induces a monotone function on the Lindenbaum algebra of $T$). Let $[\varphi]$ denote the equivalence class of $\varphi$ modulo $T$ provable equivalence, i.e., $[\varphi] := \{ \psi : T\vdash \varphi \leftrightarrow \psi \}.$ One of the main theorems of \cite{montalban2019inevitability} is the following.


\begin{theorem}[Montalb{\'a}n--W.]\label{minimal}
Let $\mathfrak{g}$ be recursive and monotone such that:
\begin{itemize}
\item for all $\varphi$, $T + \mathsf{Con}_T(\varphi) \vdash \mathfrak{g}(\varphi)$
\item for all consistent $\varphi$, $T + \varphi \nvdash \mathfrak{g}(\varphi)$
\end{itemize}
Then for every true $\varphi$, there is a true $\psi$ such that $T + \psi \vdash \varphi$ and 
$$[\psi \wedge \mathfrak{g}(\psi)] = [\psi \wedge \mathsf{Con}_T(\psi)].$$
\end{theorem}

To state a corollary of this theorem, we recall that $\varphi$ \emph{strictly implies} $\psi$ if one of the following holds:
\begin{itemize}
\item[(i)] $T+\varphi\vdash\psi$ and $T+\psi\nvdash\varphi$.
\item[(ii)] $[\varphi]=[\psi]=[\bot]$.
\end{itemize}

\begin{corollary}
There is no recursive monotone $\mathfrak{g}$ such that for every $\varphi$,\\ $\big(\varphi\wedge \mathsf{Con}_T(\varphi)\big)$ strictly implies $\big(\varphi \wedge \mathfrak{g}(\varphi) \big)$ and $\big(\varphi \wedge \mathfrak{g}(\varphi) \big)$ strictly implies $\varphi$.
\end{corollary}

The Slaman--Steel theorem suggests a strengthening of these results. Recall that a \emph{cone} in the Turing degrees is any set of the form $\{ B : B\geq_T A \}$ where $A$ is a Turing degree. The following is a special case of a theorem due to Slaman and Steel.

\begin{theorem}[Slaman--Steel]
Let $f:2^\omega\rightarrow 2^\omega$ be Borel and order-preserving. Then one of the following holds:
\begin{enumerate}
\item $f(A) \leq_T A$ on a cone.
\item $A' \leq_T f(A)$ on a cone.
\end{enumerate}
\end{theorem}

Montalb{\'a}n and the author asked whether Theorem \ref{minimal} could be strengthened in the style of the Slaman--Steel theorem, i.e., by showing that all increasing monotone recursive functions that are no stronger than the consistency operator are equivalent to the consistency operator in the limit. In this note we provide a positive answer to this question. 

To sharpen the notion of the ``limit behavior'' of a function, we introduce the notion of a true cone. A \emph{cone} is any set $\mathfrak{C}$ of the form $\{ \psi: T+\psi\vdash\varphi\}$ where $\varphi$ is a sentence. A \emph{true cone} is a cone that contains a true sentence. In \textsection{2} we prove that all recursive monotone operators that produce sentences of some bounded arithmetical complexity are bounded from below by the consistency operator on a true cone.

\begin{theorem}\label{main first}
Let $\mathfrak{g}$ be recursive and monotone such that, for some $k\in \mathbb{N}$, for all $\varphi$, $\mathfrak{g}(\varphi)$ is $\Pi_k$. Then one of the following holds:
\begin{enumerate}
\item There is a true cone $\mathfrak{C}$ such that for all $\varphi \in \mathfrak{C}$, 
$$T+ \varphi \vdash \mathfrak{g}(\varphi) .$$
\item There is a true cone $\mathfrak{C}$ such that for all $\varphi \in \mathfrak{C}$, 
$$T+ \varphi + \mathfrak{g}(\varphi) \vdash \mathsf{Con}_T(\varphi).$$
\end{enumerate} 
\end{theorem}

In \textsection{3} we prove that the condition that $\mathfrak{g}$ is recursive cannot be weakened. More precisely, we exhibit a monotone $0'$ recursive function which vacillates between behaving like the identity operator and behaving like the consistency operator.
 
 \begin{theorem}\label{main second}
There is a $0'$ recursive monotone function $\mathfrak{g}$ such that, for every $\varphi$, $\mathfrak{g}(\varphi)$ is $\Pi_1$, yet for arbitrarily strong true sentences 
$$[\varphi \wedge \mathfrak{g}(\varphi)]=[\varphi\wedge\mathsf{Con}_T(\varphi)]$$
and for arbitrarily strong true sentences
$$[\varphi \wedge \mathfrak{g}(\varphi)]=[\varphi].$$
\end{theorem}

Though Theorem \ref{main first} is a considerable strengthening of the result in \cite{montalban2019inevitability}, we conjecture that it admits of a dramatic improvement. We remind the reader that the aforementioned theorem of Slaman and Steel is a special case of a sweeping classification of increasing Borel order-preserving function. We say that a function $f$ is \emph{increasing} if, for all $A$, $A\leq_Tf(A)$.

\begin{theorem}[Slaman--Steel]
Let $f:2^\omega\rightarrow 2^\omega$ be increasing, Borel, order-preserving. Suppose that for some $\alpha<\omega_1$, $f(A) \leq_T A^{(\alpha)}$ for every $A$. Then for some $\beta\leq\alpha$, $f(A)\equiv_T A^{(\beta)}$ on a cone.
\end{theorem}

We conjecture that a similar classification of monotone proof-theoretic operators is possible. Our conjecture is stated in terms of iterated consistency statements. Let $\prec$ be a nice elementary presentation of a recursive well-ordering.\footnote{Nice elementary presentations of well-orderings are defined in \cite{beklemishev1995iterated}, see \textsection 2.3, Definition 1.} We define the iterates of the consistency operator by appealing to G{\"o}del's fixed point lemma.
$$T \vdash \mathsf{Con}^\alpha_T(\varphi) \leftrightarrow \forall \beta \prec \alpha \mathsf{Con}_T \big(\varphi\wedge\mathsf{Con}^\beta_T(\varphi) \big)$$
For true $\varphi$, the iterations of $\mathsf{Con}_T$ form a proper hierarchy of true sentences by G{\"o}del's second incompleteness theorem. We make the following conjecture.

\begin{conjecture}
Suppose $\mathfrak{g}$ is monotone, non-constant, and recursive such that, for every $\varphi$, $\mathfrak{g}(\varphi)\in\Pi_1$. Let $\prec$ be a nice elementary presentation of well-ordering and $\alpha$ an ordinal notation. Suppose that, for every $\varphi$, $$T+\varphi + \mathsf{Con}_T^\alpha(\varphi) \vdash \mathfrak{g}(\varphi).$$ Then for some $\beta\preceq\alpha$, for all $\varphi$ in a true cone, $$[\varphi\wedge \mathfrak{g}(\varphi)]=[\varphi\wedge \mathsf{Con}_T^\beta(\varphi)].$$
\end{conjecture}

According to the conjecture, if an increasing monotone recursive function $\mathfrak{g}$ that produces only $\Pi_1$ sentences is no stronger than $\mathsf{Con}_T^\alpha$, it is equivalent on a true cone to $\mathsf{Con}_T^\beta$ for some $\beta\preceq\alpha$. This would provide a classification of a large class of monotone proof-theoretic operators in terms of their limit behavior.


\section{The main theorem}

Let $T$ be a sound, recursively axiomatized extension of $\mathsf{EA}$ in the language of arithmetic. We want to show that $T+\mathsf{Con}_T(\varphi)$ is the weakest natural theory that results from adjoining a $\Pi_1$ sentence to $T$. A central notion in our approach is that of a monotone operator on finite extensions of $T$.

\begin{definition}
$\mathfrak{g}$ is \emph{monotone} if, for every $\varphi$ and $\psi$, $$T\vdash\varphi \rightarrow \psi \textrm{ implies }T\vdash \mathfrak{g}(\varphi)\rightarrow \mathfrak{g}(\psi).$$
\end{definition}

\begin{remark}
We will switch quite frequently using the notation $T+\varphi \vdash \psi$ and $T\vdash \varphi \rightarrow \psi$, trusting that no confusion arises. The two claims are equivalent, by the Deduction Theorem.
\end{remark}

Our goal is to prove that the consistency operator is, roughly, the weakest operator for uniformly strengthening theories. Our strategy is to show that any uniform method for extending theories that is as weak as the consistency operator must be equivalent to the consistency operator in the limit. We sharpen the notion ``in the limit'' with the following definitions.

\begin{definition}
Given a sentence $\varphi$, the \emph{cone generated by $\varphi$} is the set of all sentences $\psi$ such that $T\vdash \psi \rightarrow \varphi$. A \emph{cone} is any set $\mathfrak{C}$ such that, for some $\varphi$, $\mathfrak{C}$ is the cone generated by $\varphi$. A \emph{true cone} is a cone that is generated by a sentence that is true in the standard model $\mathbb{N}$.
\end{definition}

We are now ready to state and prove the main theorem. Note that the following is a restatement of Theorem \ref{main first}.

\begin{theorem}\label{main} Let $T$ be a sound, effectively axiomatized extension of $\mathsf{EA}$. Let $\mathfrak{g}$ be recursive and monotone such that, for some $k\in\mathbb{N}$, for all $\varphi$, $\mathfrak{g}(\varphi)$ is $\Pi_k$. Then one of the following holds:
\begin{enumerate}
\item There is a true cone $\mathfrak{C}$ such that for all $\varphi \in \mathfrak{C}$, 
$$T+\varphi \vdash \mathfrak{g}(\varphi).$$
\item There is a true cone $\mathfrak{C}$ such that for all $\varphi \in \mathfrak{C}$, 
$$T+ \varphi + \mathfrak{g}(\varphi) \vdash \mathsf{Con}_T(\varphi).$$
\end{enumerate} 
\end{theorem}

\begin{proof}
Since $\mathfrak{g}$ is recursive, its graph is defined by a $\Sigma_1$ formula $\mathcal{G}$, i.e., for any $\varphi$ and $\psi$,
$$\mathfrak{g}(\varphi)=\psi \iff \mathbb{N}\vDash \mathcal{G}(\ulcorner\varphi\urcorner,\ulcorner\psi\urcorner).$$
Since $T$ is sound and $\Sigma_1$ complete, this implies that for any $\varphi$ and $\psi$,
\begin{equation} 
\label{eq:star} 
\tag{$\star$} 
\mathfrak{g}(\varphi)=\psi \iff T\vdash \mathcal{G}(\ulcorner\varphi\urcorner,\ulcorner\psi\urcorner)
\end{equation} 
From now on we drop the corner quotes and write $\mathcal{G}(\varphi,\psi)$ instead of $\mathcal{G}(\ulcorner\varphi\urcorner,\ulcorner\psi\urcorner)$, trusting that no confusion will arise.



We consider the following sentence in the language of arithmetic:
\begin{equation} 
\label{eq:A} 
\tag{$A$} 
\forall x \forall y \Big( \big( \mathcal{G}(x,y) \wedge \mathsf{True}_{\Pi_k} (y) \big) \rightarrow \mathsf{Con}_T(x) \Big).
\end{equation} 

Informally, $A$ says that, for every $\varphi$, the truth of $\mathfrak{g}(\varphi)$ implies the consistency of $T+ \varphi$. Note that we need to use a partial truth predicate in the statement $A$ since we are formalizing a uniform claim about the outputs of the function $\mathfrak{g}$. For any specific output $\psi$ of the function $\mathfrak{g}$, $T$ will be able to reason about $\psi$ without relying on the partial truth predicate.


We break into cases based on whether $A$ is true or false.

\textbf{Case 1:} $A$ is true in the standard model $\mathbb{N}$. We claim that in this case $$\mathfrak{C}:=\{ \varphi : T + \varphi \vdash A \}$$ satisfies condition (2) from the statement of the theorem. Clearly $\mathfrak{C}$ is a true cone. It suffices to show that for any $\varphi \in \mathfrak{C}$, $$T+ \varphi + \mathfrak{g}(\varphi) \vdash \mathsf{Con}_T(\varphi).$$
 So let $\varphi\in\mathfrak{C}$ and let $\psi=\mathfrak{g}(\varphi)$. We reason as follows.
\begin{flalign*}
T + \varphi & \vdash \forall x \forall y \Big( \big( \mathcal{G}(x,y) \wedge \mathsf{True}_{\Pi_k} (y) \big) \rightarrow \mathsf{Con}_T(x) \Big) \textrm{ by choice of $\mathfrak{C}$.}\\
T + \varphi & \vdash \big( \mathcal{G}(\varphi ,  \psi ) \wedge \mathsf{True}_{\Pi_k}(\psi)\big) \rightarrow \mathsf{Con}_T(\varphi)  \textrm{ by instantiation.}\\
T + \varphi & \vdash \mathcal{G}( \varphi ,  \psi) \textrm{ by observation $(\star)$.}\\
T + \varphi & \vdash  \mathsf{True}_{\Pi_k}(\psi) \rightarrow \mathsf{Con}_T(\varphi)  \textrm{ from the previous two lines by logic.}\\
T + \varphi + \psi &\vdash \mathsf{Con}_T(\varphi) \textrm{ trivially from the previous line.}\\
T+ \varphi + \mathfrak{g}(\varphi) & \vdash \mathsf{Con}_T(\varphi) \textrm{ since $\psi=\mathfrak{g}(\varphi)$.}
\end{flalign*}

\textbf{Case 2:} $A$ is false in the standard model $\mathbb{N}$. We infer that 
$$ \exists \varphi \exists \psi \Big(  \mathcal{G}(\varphi,\psi) \wedge \mathsf{True}_{\Pi_k} (\psi)  \wedge \neg \mathsf{Con}_T(\varphi) \Big).$$
Thus, there is an inconsistent sentence $\varphi$ such that $\mathfrak{g}(\varphi)$ is a true $\Pi_k$ sentence. This is to say that $\mathfrak{g}(\bot)$ is true. We claim that in this case $$\mathfrak{C}:=\{ \varphi : T + \varphi \vdash \mathfrak{g}(\bot) \}$$ satisfies condition (1) from the statement of the theorem. Clearly $\mathfrak{C}$ is a true cone. It suffices to show that for any $\varphi \in \mathfrak{C}$, $$T+ \varphi \vdash \mathfrak{g}(\varphi).$$ So let $\varphi \in \mathfrak{C}$. By the definition of $\mathfrak{C}$, we infer that
\begin{equation} 
\label{eq:dagger} 
\tag{$\dagger$} 
T+\varphi \vdash \mathfrak{g}(\bot).
\end{equation} 

We reason as follows.
\begin{flalign*}
T&\vdash \bot \rightarrow \varphi \textrm{ by logic.} \\
T&\vdash \mathfrak{g}(\bot) \rightarrow  \mathfrak{g}(\varphi) \textrm{ by the monotonicity of $\mathfrak{g}$.}\\
T+\varphi&\vdash \mathfrak{g}(\varphi) \textrm{ from the previous line and $\dagger$, by logic.}
\end{flalign*}
This completes the proof.
\end{proof}

\begin{remark}
Theorem \ref{main} is stated about operators $\mathfrak{g}$ that produce sentences of bounded arithmetical complexity, i.e., for some $k\in\mathbb{N}$, for all $\varphi$, $\mathfrak{g}(\varphi)$ is $\Pi_k$. The reason for this restriction is to invoke the partial truth-predicate for $\Pi_k$ sentences when providing the sentence $A$.
\end{remark}


\section{Recursiveness is a necessary condition}

In the proof of Theorem \ref{main} we appealed to the recursiveness of $\mathfrak{g}$ to show that $T$ correctly calculates the values of $\mathfrak{g}$, i.e., that for every $\varphi$ and $\psi$,
$$\mathfrak{g}(\varphi)=\psi \iff T\vdash \mathcal{G}(\ulcorner\varphi\urcorner,\ulcorner\psi\urcorner).$$ In this section we show that recursiveness is a necessary condition for the proof of Theorem \ref{main}. In particular, we exhibit a monotone operator $\mathfrak{g}$ which is recursive in $0'$ and produces only $\Pi_1$ sentences such that for arbitrarily strong true sentences $\varphi$, $$[\varphi \wedge \mathfrak{g}(\varphi)]=[\varphi\wedge\mathsf{Con}_T(\varphi)]$$
and for arbitrarily strong true sentences $\varphi$,
$$[\varphi \wedge \mathfrak{g}(\varphi)]=[\varphi].$$
Our proof makes use of a recursive set $\mathfrak{A}$ that contains arbitrarily strong true sentences and omits arbitrarily strong true sentences. A very similar set is constructed in \cite{montalban2019inevitability}. We now present the construction of the set $\mathfrak{A}$, which is necessary to understand the proof of the theorem. After describing the construction of $\mathfrak{A}$ we will verify some of its basic properties.

\begin{quote}

Let $\{\varphi_0,\varphi_1,...\}$ be an effective G{\"o}del numbering of the language of arithmetic. We describe the construction of $\mathfrak{A}$ in stages. During a stage $n$ we may \emph{activate} a sentence $\psi$, in which case we say that $\psi$ is \emph{active} until it is \emph{deactivated} at some later stage.

\textbf{Stage 0:} Numerate $\varphi_0$ and $\neg\varphi_0$ into $\mathfrak{A}$. Activate the sentences $(\varphi_0\wedge \mathsf{Con}_T(\varphi_0))$ and $(\neg\varphi_0\wedge \mathsf{Con}_T(\neg\varphi_0))$.

\textbf{Stage n+1:} There are finitely many active sentences. For each such sentence $\psi$, numerate $\theta_0:=(\psi\wedge\varphi_{n+1})$ and $\theta_1:=(\psi\wedge\neg\varphi_{n+1})$ into $\mathfrak{A}$. Deactivate the sentence $\psi$ and activate the sentences $(\theta_0\wedge \mathsf{Con}_T(\theta_0))$ and $(\theta_1\wedge \mathsf{Con}_T(\theta_1))$.
\end{quote}

\begin{remark}\label{branching}
It can be useful to visualize, along with the construction of $\mathfrak{A}$, the construction of an upwards growing tree that is (at most) binary branching. The nodes in the tree are the \emph{consistent} sentences that are numerated into $\mathfrak{A}$. The immediate successors in this tree of a sentence $\varphi$ have the form $(\varphi \wedge \mathsf{Con}_T(\varphi) \wedge \theta)$ and $(\varphi \wedge \mathsf{Con}_T(\varphi) \wedge \neg \theta)$. Thus, the successors of any two points are inconsistent with each other. Observe that for any two distinct sentences $\varphi$ and $\psi$ in the tree, $\varphi$ is below $\psi$ (i.e., $\varphi$ and $\psi$ belong to the same path and $\varphi$ is below $\psi$) if and only if $T+\psi \vdash \varphi$. It follows from the previous two observations that any two sentences that are incompatible with each other in the tree ordering are inconsistent with each other.
\end{remark}



\begin{lemma}\label{truth lemma}
At any stage in the construction of $\mathfrak{A}$, (i) exactly one of the active sentences is true in the standard model $\mathbb{N}$ and (ii) exactly one of the sentences numerated into $\mathfrak{A}$ is true in the standard model $\mathbb{N}$.
\end{lemma}

\begin{proof}
We proceed by induction on the stages in the construction of $\mathfrak{A}$.

\textbf{Stage 0:}  Exactly one of $\varphi_0$ or $\neg\varphi_0$ is true (these are the numerated sentences), and hence so is exactly one of $\varphi_0\wedge \mathsf{Con}(\varphi_0)$ and $\neg\varphi_0\wedge \mathsf{Con}(\neg\varphi_0)$ (these are the activated sentences). 

\textbf{Stage n+1:} At the end of stage $n$ there is exactly one true activated sentence $\theta$. Then so exactly one of $\zeta_0:=\theta\wedge\varphi_n$ and $\zeta_1:=\theta\wedge\neg\varphi_n$ is true (these are the numerated sentences). Hence exactly one of $\zeta_0\wedge \mathsf{Con}(\zeta_0)$ and $\zeta_1\wedge \mathsf{Con}(\zeta_1)$ is true (these are the activated sentences).
\end{proof}

\begin{corollary}\label{true branch}
There is a unique branch through the tree described in Remark \ref{branching} that contains only true sentences. We will call it the \emph{true branch.}
\end{corollary}

\begin{lemma}\label{arbitrary strong}
$\mathfrak{A}$ contains arbitrarily strong true sentences.
\end{lemma}

\begin{proof}
Let $\psi$ be a true sentence. $\psi$ appears at some point in our G{\"o}del numbering of the language of arithmetic, i.e., for some $n$, $\psi$ is $\varphi_n$. Going into stage $n$ of the construction of $\mathfrak{A}$, there is exactly one true active sentence $\theta$ by Lemma \ref{truth lemma}. Then $\theta \wedge \varphi_n$ is numerated into $\mathfrak{A}$. So $\mathfrak{A}$ contains a true sentence that implies $\psi$. 
\end{proof}

Our proof also makes use of iterated consistency statements. Let $\prec$ be an elementary presentation of $\omega$. For the sake of convenience, we reiterate the definition of the iterates of the consistency operator. We define these iterates by appealing to G{\"o}del's fixed point lemma:
$$T \vdash \mathsf{Con}^\alpha_T(\varphi) \leftrightarrow \forall \beta \prec \alpha \mathsf{Con}_T \big(\varphi\wedge\mathsf{Con}^\beta_T(\varphi) \big)$$
For true $\varphi$, the iterates of $\mathsf{Con}_T$ form a proper hierarchy of true sentences by G{\"o}del's second incompleteness theorem.

\begin{definition}
For a true sentence $\psi$ numerated into $\mathfrak{A}$ at stage $n$, let $\theta_\psi$ be a true sentence that is either the $(n+1)^{th}$ sentence in the G{\"o}del numbering of the language or the negation thereof (depending on which is true). The point of the definition is this: if $\psi$ is a true sentence numerated into $\mathfrak{A}$, then the next true sentence numerated into $\mathfrak{A}$ is $\psi \wedge \mathsf{Con}_T(\psi) \wedge \theta_\psi$.
\end{definition}


\begin{lemma}
\label{arbitrarily strong}
For arbitrarily strong true sentences $\psi \in \mathfrak{A}$, $T \nvdash (\psi \wedge \mathsf{Con}_T(\psi)) \rightarrow \theta_\psi.$
\end{lemma}

\begin{proof}
Suppose not, i.e., suppose that there is a true $\varphi$ such that for all true $\psi$, if both $T\vdash \psi \rightarrow \varphi$ and $\psi \in \mathfrak{A}$, then:
\begin{equation}
\label{eq:oplus} 
\tag{$\oplus$} 
T\vdash (\psi \wedge \mathsf{Con}_T(\psi)) \rightarrow \theta_\psi.
\end{equation}
By Lemma \ref{arbitrary strong}, $\mathfrak{A}$ contains arbitrarily strong true sentences, so we know there is at least one such sentence $\psi_0$ in $\mathfrak{A}$ that implies $\varphi$. By the construction of $\mathfrak{A}$, the true sentences numerated into $\mathfrak{A}$ after $\psi_0$ are:
\begin{itemize}
\item $\psi_1:= \psi_0 \wedge \mathsf{Con}_T(\psi_0) \wedge \theta_{\psi_0}$
\item $\psi_2:= \psi_1 \wedge \mathsf{Con}_T(\psi_1) \wedge \theta_{\psi_1}$
\end{itemize}
and so on. Each $\psi_n$ implies $\varphi$ and is in $\mathfrak{A}$. Thus, each $\psi_n$ satisfies condition ($\oplus$), i.e., for each $\psi_n$, $T\vdash (\psi_n \wedge \mathsf{Con}_T(\psi_n)) \rightarrow \theta_{\psi_n}.$ This means that for all $n\geq 1$ the final conjunct of $\psi_n$ is superfluous. It follows that for each $n$, $\psi_n$ is $T$ provably equivalent to $\psi_0 \wedge \mathsf{Con}^n_T(\psi_0).$ But then no sentence in $\mathfrak{A}$ is stronger than $\psi_0 \wedge \mathsf{Con}^\omega_T(\psi_0)$, contradicting the fact proved in Lemma \ref{arbitrary strong}, i.e., that $\mathfrak{A}$ contains arbitrarily strong true sentences. 
\end{proof}

We are now ready to state and prove the theorem. Note that the following is a restatement of Theorem \ref{main second}.

\begin{theorem}
Let $T$ be a sound, effectively axiomatized extension of $\mathsf{EA}$. There is a $0'$ recursive monotone function $\mathfrak{g}$ such that, for every $\varphi$, $\mathfrak{g}(\varphi)$ is $\Pi_1$, yet for arbitrarily strong true sentences $\varphi$
$$[\varphi \wedge \mathfrak{g}(\varphi)]=[\varphi\wedge\mathsf{Con}_T(\varphi)]$$
and for arbitrarily strong true sentences $\varphi$
$$[\varphi \wedge \mathfrak{g}(\varphi)]=[\varphi].$$
\end{theorem}

\begin{proof}

We define the function $\mathfrak{g}$ as follows:
\begin{equation*}
   \mathfrak{g}(\varphi)=%
   \begin{cases}
   \bigwedge\{ \mathsf{Con}_T(\zeta) : \zeta\in\mathfrak{A} \textrm{ and } T+\varphi \vdash \zeta \} &\text{if $[\varphi]\neq[\bot]$}\\
     \bot &\text{otherwise}
   \end{cases}
\end{equation*}

We will check one-by-one that $\mathfrak{g}$ satisfies the properties ascribed to it in the statement of the theorem. We start by checking that $\mathfrak{g}$ is $0'$ recursive. In so doing, we will also demonstrate that $\mathfrak{g}$ is well defined, i.e., always produces a finitary sentence.

\begin{claim}
$\mathfrak{g}$ is $0'$ recursive.
\end{claim}

To verify that $\mathfrak{g}$ is $0'$ recursive, we informally describe an algorithm for calculating $\mathfrak{g}$ using $0'$ as an oracle. Here is the algorithm: Given an input $\varphi$, first use $0'$ to determine whether $[\varphi]=[\bot]$. If so, output $\bot$. Otherwise, we have to find all sentences $\psi \in \mathfrak{A}$ such that $T+\varphi \vdash \zeta$. Let's say that $\varphi$ is the $n^{th}$ sentence in our G{\"o}del numbering of the language of arithmetic. By the construction of $\mathfrak{A}$, the only sentences in $\mathfrak{A}$ that $T+\varphi$ proves must have been numerated into $\mathfrak{A}$ by stage $n$. So find each of the finitely many sentences that were activated by stage $n$ in the construction of $\mathfrak{A}$. For any such sentence $\zeta$, use $0'$ to determine whether $T+\varphi \vdash \zeta$.\footnote{Querying $0'$ is not strictly necessary here. If we already know that two sentences $\varphi$ and $\zeta$ are consistent, we can determine effectively whether $T+\varphi \vdash \zeta$ by paying attention to details of the construction of $\mathfrak{A}$.} Once all sentences $\psi \in \mathfrak{A}$ such that $T+\varphi \vdash \zeta$ have been found, output the conjunction of their consistency statements.

It is now routine to verify that the following claim is true:

\begin{claim}
$\mathfrak{g}$ is monotone and always produces a $\Pi_1$ sentence. 
\end{claim}

We now work towards showing that $\mathfrak{g}$ behaves like the consistency operator for arbitrarily strong inputs.

\begin{claim}
For arbitrarily strong true sentences $\varphi$, $[\varphi \wedge \mathfrak{g}(\varphi)]=[\varphi\wedge\mathsf{Con}_T(\varphi)].$
\end{claim}

To see why the claim is true, note that whenever $\varphi\in\mathfrak{A}$, it follows that $[\mathfrak{g}(\varphi)]=[\mathsf{Con}_T(\varphi)]$, whence 
$$[\varphi \wedge \mathfrak{g}(\varphi)]=[\varphi\wedge\mathsf{Con}_T(\varphi)].$$
Since $\mathfrak{A}$ contains arbitrary strong true sentences, it follows immediately that for arbitrarily strong true sentences $\varphi$, $$[\varphi \wedge \mathfrak{g}(\varphi)]=[\varphi\wedge\mathsf{Con}_T(\varphi)].$$ 

Thus, to prove the theorem, it suffices to see that $\mathfrak{g}$ behaves like the identity operator on arbitrarily strong inputs.

\begin{claim}
\label{goal}
For arbitrarily strong true sentences $\varphi$, $[\varphi \wedge \mathfrak{g}(\varphi)]=[\varphi].$
\end{claim}

To this end, we will assume only that $\psi$ is a true sentence satisfying the following claim:
\begin{equation}
\label{eq:sharp} 
\tag{$\sharp$} 
T\nvdash (\psi \wedge \mathsf{Con}_T(\psi)) \rightarrow \theta_\psi.
\end{equation}
Recall that if $\psi$ was numerated into $\mathfrak{A}$ at stage $n$, then $\theta_\psi$ is a true sentence that is either the $(n+1)^{th}$ sentence in the G{\"o}del numbering of the language or the negation thereof (depending on which is true). By Lemma \ref{arbitrarily strong}, we know that for arbitrarily strong true sentences $\psi$, $\psi$ satisfies $(\sharp)$. We will then show that for any such $\psi$, where $\varphi$ is the sentence $(\psi \wedge \mathsf{Con}_T(\psi))$, the following identity holds: $$[\varphi \wedge \mathfrak{g}(\varphi)]=[\varphi],$$ thus certifying the truth of Claim \ref{goal}.

So let $\psi$ be a true sentence in $\mathfrak{A}$ satisfying ($\sharp$). By Lemma \ref{truth lemma}, there is a unique next true sentence numerated into $\mathfrak{A}$, and that sentence is $\big(\psi \wedge \mathsf{Con}_T(\psi)\wedge\theta_\psi \big)$. We assert the following claim:

\begin{claim}\label{big claim}
For all sentences $\zeta \in \mathfrak{A}$, if $T +\big(\psi \wedge \mathsf{Con}_T(\psi) \big) \vdash \zeta$ then also $T+\psi\vdash \zeta$.
\end{claim}

Let's see why Claim \ref{big claim} is true. Since $(\psi \wedge\mathsf{Con}_T(\psi))$ is true, any sentence $\zeta$ in $\mathfrak{A}$ that is implied by $\big(\psi \wedge \mathsf{Con}_T(\psi) \big)$ belongs to the true branch (see Corollary \ref{true branch}). By assumption $(\sharp)$, $(\psi \wedge \mathsf{Con}_T(\psi))$ has strength strictly intermediate between $\psi$ and $(\psi \wedge \mathsf{Con}_T(\psi) \wedge \theta_\psi)$. Accordingly, any sentence $\zeta$ in $\mathfrak{A}$ that is implied by $(\psi \wedge\mathsf{Con}_T(\psi))$ must have been numerated into $\mathfrak{A}$ before $(\psi \wedge \mathsf{Con}_T(\psi) \wedge \theta_\psi)$. Recall that $\psi$ is the sentence in the true branch numerated into $\mathfrak{A}$ immediately before $(\psi \wedge \mathsf{Con}_T(\psi) \wedge \theta_\psi)$. So $\zeta$ either is $\psi$ or was numerated into $\mathfrak{A}$ earlier than $\psi$. Either way, $\psi$ implies $\zeta$. This certifies the truth of Claim \ref{big claim}.

We now introduce the formula $\varphi := \big( \psi \wedge \mathsf{Con}_T(\psi) \big)$. We make the following claim:

\begin{claim}
\label{mini claim}
$[\varphi \wedge \mathfrak{g}(\varphi)] = [\varphi \wedge \bigwedge\{ \mathsf{Con}_T(\zeta) : \zeta\in\mathfrak{A} \textrm{ and } T + \psi \vdash \zeta \}]$.
\end{claim}

We argue for Claim \ref{mini claim} as follows:
\begin{flalign*}
[\varphi \wedge \mathfrak{g}(\varphi)] &= [\varphi \wedge \bigwedge\{ \mathsf{Con}_T(\zeta) : \zeta\in\mathfrak{A} \textrm{ and } T + \varphi \vdash \zeta \}] \textrm{ by definition of $\mathfrak{g}$}\\
&= [\varphi \wedge \bigwedge\{ \mathsf{Con}_T(\zeta) : \zeta\in\mathfrak{A} \textrm{ and } T + (\psi\wedge \mathsf{Con}_T(\psi)) \vdash \zeta \}] \textrm{ by choice of $\varphi$}\\
&= [\varphi \wedge \bigwedge\{ \mathsf{Con}_T(\zeta) : \zeta\in\mathfrak{A} \textrm{ and } T + \psi \vdash \zeta \}] \textrm{ by Claim \ref{big claim}.}
\end{flalign*}
With Claim \ref{mini claim} on board, we are now ready to prove that $[\varphi]=[\varphi \wedge \mathfrak{g}(\varphi)].$ We reason as follows:
\begin{flalign*}
 T + \varphi + \mathsf{Con}_T(\psi) &\vdash \bigwedge\{ \mathsf{Con}_T(\zeta) : \zeta\in\mathfrak{A} \textrm{ and } T + \psi \vdash \zeta\} \textrm{ by the monotonity of $\mathsf{Con}_T$}. \\
 T + \varphi + \mathsf{Con}_T(\psi) &\vdash  \mathfrak{g}(\varphi) \textrm{ by Claim \ref{mini claim}.}\\
 T + \varphi & \vdash  \mathfrak{g}(\varphi) \textrm{ since $T+\varphi \vdash  \mathsf{Con}_T(\psi)$, by the choice of $\varphi$.}
 \end{flalign*}
This trivially implies that: $$[\varphi]=[\varphi \wedge \mathfrak{g}(\varphi)].$$
This completes the proof of the theorem.
\end{proof}

\bibliographystyle{plain}
\bibliography{thebibliography}

\end{document}